\documentclass[11pt]{article}

\usepackage[margin=1in]{geometry}
\usepackage{amsmath,amssymb,amsthm,mathtools}
\usepackage{enumitem}
\usepackage{microtype}
\usepackage{xcolor}
\usepackage{booktabs}
\usepackage{tikz}
\usepackage{hyperref}
\usepackage[nameinlink,noabbrev]{cleveref}
\usepackage{authblk}

\hypersetup{
  colorlinks=true,
  linkcolor=blue!55!black,
  citecolor=blue!55!black,
  urlcolor=blue!55!black,
  pdftitle={Discounted Hitting Domination on Graphs},
  pdfauthor={Julian D. Allagan, Kevin Pereyra, and William A. Massey}
}

\newtheorem{theorem}{Theorem}[section]
\newtheorem{proposition}[theorem]{Proposition}
\newtheorem{lemma}[theorem]{Lemma}
\newtheorem{corollary}[theorem]{Corollary}

\theoremstyle{definition}
\newtheorem{definition}[theorem]{Definition}
\newtheorem{example}[theorem]{Example}
\newtheorem{problem}[theorem]{Problem}
\theoremstyle{remark}
\newtheorem{remark}[theorem]{Remark}

\DeclareMathOperator{\dist}{dist}
\DeclareMathOperator{\arcosh}{arcosh}
\DeclareMathOperator{\sech}{sech}
\DeclareMathOperator{\Sp}{Sp}

\newcommand{\N}{\mathbb{N}}
\newcommand{\Q}{\mathbb{Q}}
\newcommand{\one}{\mathbf{1}}
\newcommand{\pospart}[1]{\left(#1\right)_{+}}

\title{Discounted Hitting Domination on Graphs with Submodularity, Complexity and Exact Algorithms}

\author[1]{Julian D.\ Allagan}
\author[2]{Kevin Pereyra}
\author[3]{William A.\ Massey}
\affil[1]{Department of Mathematics, University of Maryland Eastern Shore, Princess Anne, MD 21853, USA\\ \href{mailto:jallagan@umes.edu}{jallagan@umes.edu}}
\affil[2]{Departamento de Matem\'atica, Universidad Nacional de San Luis, San Luis 5700, Argentina\\ \href{mailto:kdpereyra@unsl.edu.ar}{kdpereyra@unsl.edu.ar}}
\affil[3]{Department of Operations Research and Financial Engineering, Princeton University, Princeton, NJ 08544, USA\\ \href{mailto:wmassey@princeton.edu}{wmassey@princeton.edu}}
\date{}

\begin{document}
\maketitle

\begin{abstract}
On a network with a fixed set of verified sources, discounted averaging induces an equilibrium support
$h_i^S=\mathbb{E}_i[\lambda^{T_S}]$,
the discounted probability that a random walk reaches $S$ before attenuation. We define the
\emph{discounted hitting domination number} $\delta_{\lambda,\tau}(G)$ as the minimum number of
sources required to guarantee $h_i^S\ge\tau$ at every vertex. Although this potential is known through
penalized and group hitting probabilities, the associated minimum-cardinality uniform-coverage problem
appears to be new. Aggregate support is monotone submodular, while the uniform-floor problem is an exact
submodular-cover problem. Moreover, if
$\lambda^{r+1}<\tau\le\left(\frac{\lambda}{\Delta}\right)^r$, then $\delta_{\lambda,\tau}(G)$ equals the distance-$r$ domination number. This yields NP-completeness
and APX-completeness at $(\lambda,\tau)=(1/4,1/14)$ on graphs of maximum degree three. For spiders, we
obtain an exact finite-state characterization and a polynomial-time algorithm for every fixed rational pair
$(\lambda,\tau)$, and show that the branching vertex need not belong to a minimum source set. Finally, an
exact mixed-integer linear formulation certifies optimal placements on a real network and a synthetic graph
and demonstrates substantial differences from degree, closeness, and classical domination.
\end{abstract}


\section{Introduction}\label{sec:intro}

Let $S$ be a set of pinned, verified vertices in a network whose remaining
vertices repeatedly average the information received from their neighbors,
with a fidelity factor $\lambda\in(0,1)$ applied at each step. The resulting
equilibrium support at a vertex $i$ is $h_i^S=\mathbb{E}_i[\lambda^{T_S}]$,
where $T_S$ is the first hitting time of $S$ by the associated random walk.
We ask for the smallest source set for which every vertex has support at least
a prescribed threshold $\tau$. This leads to the invariant
$\delta_{\lambda,\tau}(G)$, which we call the
\emph{discounted hitting domination number}; a source set attaining the floor
is a \emph{$(\lambda,\tau)$-dominating set}. The terminology distinguishes this
placement parameter from the enumerative invariant $\zeta(G)$, the number of
minimum dominating sets of a graph, historically called the dominion of $G$
\cite{allagan2021,su2024,allagan2026}.

The potential $h_i^S$ is not new in itself. It is a multiplicatively discounted
first-hitting quantity and coincides, under the corresponding conventions, with
penalized hitting probability, group hitting probability, and a monotone
transform of discounted hitting time
\cite{sarkar2010,wu2014,wu2016,guo2024}. It is also closely related to
random-walk proximity measures and to linear consensus models with pinned or
stubborn agents
\cite{degroot1974,friedkin1990,ghaderi2014,hunter2022,fouss2007,fouss2016}.
Pinning-control problems likewise ask which nodes should be fixed, but their
objective is typically synchronization or trajectory control rather than
minimum-cardinality coverage under a uniform discounted hitting constraint
\cite{xing2019}. Related random-walk optimization problems evaluate a prescribed
target set or rank candidate targets
\cite{li2014,mavroforakis2016,guo2024}. To our knowledge, minimizing the size of
a target set subject to $h_i^S\ge\tau$ for every vertex $i$ has not previously
been studied as a graph-domination problem.

The resulting parameter connects probabilistic hitting structure with classical
domination and combinatorial optimization. We show that aggregate support is a
monotone submodular set function and that the uniform-floor problem admits an
exact submodular-cover formulation and an exact mixed-integer linear
formulation. For graphs of maximum degree $\Delta$, the parameter recovers
distance-$r$ domination whenever
$\lambda^{r+1}<\tau\le(\lambda/\Delta)^r$. In particular, at
$(\lambda,\tau)=(1/4,1/14)$ the decision problem is NP-complete and the
minimum-cardinality optimization problem is APX-complete even on graphs of
maximum degree three. For spiders, we derive an exact finite-state
characterization based on transfer data along the legs and obtain a
polynomial-time algorithm for every fixed rational pair $(\lambda,\tau)$.
The branching vertex need not belong to a minimum source set. We also use the
mixed-integer formulation to certify optimal placements on a real network and
a synthetic graph, where the discounted-floor optimum can differ substantially
from degree and closeness rankings and from minimum dominating sets.

The remainder of the paper is organized as follows. \Cref{sec:model} develops
the pinned discounted dynamics and the hitting-time representation.
\Cref{sec:structural} establishes the distance bounds and the recovery of
distance domination, while \cref{sec:optimization} gives the submodular and
mixed-integer formulations and the complexity results.
\Cref{sec:attenuation,sec:spiders} develop the transfer method and the exact
algorithm on spiders, \cref{sec:dense-sparse} treats complete graphs and stars,
and \cref{sec:experiments} presents the computational comparisons.
\section{Pinned discounted trust dynamics}\label{sec:model}

Let $V=\{1,\dots,n\}$ be a finite agent set, and let
$W=(w_{ij})_{i,j\in V}$ be a row-stochastic matrix, with $w_{ij}\ge0$ and
$\sum_j w_{ij}=1$. We interpret $w_{ij}$ as the share of agent $i$'s trust
assigned to agent $j$; the associated directed trust graph has an arc
$i\to j$ whenever $w_{ij}>0$. Let
$\Lambda=\operatorname{diag}(\lambda_1,\dots,\lambda_n)$ with
$0<\lambda_i<1$, and write $\lambda_*:=\max_i\lambda_i<1$.

\begin{definition}[Pinned discounted dynamics]\label{def:dynamics}
Let $\varnothing\neq S\subseteq V$ be a source set. Vertices in $S$ are pinned
at unit support. For $x(0)\in[0,1]^V$ with $x_i(0)=1$ for $i\in S$, define
\begin{equation}\label{eq:dynamics}
x_i(t+1)=
\begin{cases}
1, & i\in S,\\[2mm]
\lambda_i\displaystyle\sum_{j\in V}w_{ij}x_j(t), & i\notin S.
\end{cases}
\end{equation}
The coordinate $x_i(t)$ is the verification-support level at vertex $i$
after $t$ updates. The factor $\lambda_i$ is the fraction of support retained
at $i$ in one update, while $1-\lambda_i$ represents local attenuation.
No probabilistic calibration of $x_i(t)$ as a posterior probability is assumed.
\end{definition}

Let $U=V\setminus S$. After ordering the vertices in $U$ before those in $S$,
the non-source coordinates satisfy
\begin{equation}\label{eq:affine}
x_U(t+1)=B_Sx_U(t)+b_S,
\qquad
B_S:=\Lambda_UW_{UU},
\qquad
b_S:=\Lambda_UW_{US}\one.
\end{equation}

\begin{theorem}[Equilibrium and geometric convergence]\label{thm:convergence}
For every nonempty $S\subseteq V$, the dynamics \eqref{eq:dynamics} have a
unique equilibrium $h^S\in[0,1]^V$, given by
\begin{equation}\label{eq:equilibrium}
h^S_S=\one,
\qquad
h^S_U=(I-B_S)^{-1}b_S.
\end{equation}
Moreover,
$\|x_U(t)-h^S_U\|_\infty
\le\lambda_*^t\|x_U(0)-h^S_U\|_\infty$
for every $t\ge0$. If $x_U(0)=0$, then $x_U(t)$ increases coordinatewise to
$h^S_U$.
\end{theorem}

\begin{proof}
Since $W_{UU}$ is row-substochastic,
$\|B_S\|_\infty\le\|\Lambda_U\|_\infty\|W_{UU}\|_\infty\le\lambda_*<1$.
Hence $I-B_S$ is invertible, and \eqref{eq:equilibrium} is the unique fixed
point of \eqref{eq:affine}. Subtracting the fixed-point equation gives
$x_U(t)-h^S_U=B_S^t(x_U(0)-h^S_U)$, which yields the stated bound.

The map in \eqref{eq:affine} preserves $[0,1]^U$, and
$h^S_U=\sum_{t\ge0}B_S^tb_S\in[0,1]^U$. If $x_U(0)=0$, then
$x_U(1)=b_S\ge0$, and
$x_U(t+1)-x_U(t)=B_S(x_U(t)-x_U(t-1))\ge0$ inductively. The bounded
monotone sequence therefore converges, and uniqueness of the fixed point
identifies its limit with $h^S_U$.
\end{proof}

Let $(X_t)_{t\ge0}$ be the Markov chain with transition matrix $W$, and let
$T_S:=\inf\{t\ge0:X_t\in S\}$. Empty products are understood to equal one;
on $\{T_S=\infty\}$, the discounted product below is defined to be zero.

\begin{theorem}[Discounted hitting potential]\label{thm:hitting}
For every $i\in V$,
\begin{equation}\label{eq:hitting}
h_i^S
=\mathbb{E}_i\!\left[\prod_{r=0}^{T_S-1}\lambda_{X_r}\right].
\end{equation}
Under homogeneous fidelity $\lambda_i=\lambda$,
\begin{equation}\label{eq:homogeneous-hitting}
h_i^S=\mathbb{E}_i\!\left[\lambda^{T_S}\right].
\end{equation}
\end{theorem}

\begin{proof}
Let $g_i$ denote the right-hand side of \eqref{eq:hitting}. If $i\in S$, then
$T_S=0$ and $g_i=1$. If $i\notin S$, conditioning on the first transition and
using the Markov property gives
$g_i=\lambda_i\sum_jw_{ij}g_j$. Thus $g$ satisfies the same pinned fixed-point
system as $h^S$, and uniqueness from \cref{thm:convergence} gives $g=h^S$.
\end{proof}

The representation \eqref{eq:hitting} also has a killed-walk interpretation.
At vertex $i$, let the walk survive with probability $\lambda_i$ before making
a $W$-transition, and otherwise move to a cemetery state $\partial$. Then
$h_i^S$ is exactly the probability that the killed walk reaches $S$ before
$\partial$.

\begin{remark}[Identification with known proximity measures]\label{rem:known-potentials}
Under homogeneous fidelity, \eqref{eq:homogeneous-hitting} coincides with
several established random-walk quantities. For a singleton $S=\{q\}$, the
system $h_q=1$ and $h_i=\lambda\sum_jw_{ij}h_j$ for $i\neq q$ is the penalized
hitting probability of Wu, Jin, and Zhang \cite{wu2014,wu2016}, with decay
constant $c=\lambda$. For a general target set, an $\alpha$-terminating walk
that stops before each transition with probability $\alpha$ reaches $S$ before
termination with probability $\mathbb{E}_i[(1-\alpha)^{T_S}]$; setting
$\lambda=1-\alpha$ gives the group hitting probability of Guo et al.\
\cite{guo2024}.

There is also a direct relation with discounted hitting time. Let $R$ be
independent of the walk and satisfy $\mathbb{P}(R>t)=\lambda^t$ for $t\ge0$.
Then
\begin{equation*}
\mathbb{E}_i[\min\{T_S,R\}]
=\sum_{t\ge0}\lambda^t\mathbb{P}_i(T_S>t)
=\frac{1-h_i^S}{1-\lambda}.
\end{equation*}
Thus, under this convention, the discounted hitting time considered in
\cite{sarkar2010} is a decreasing affine function of $h_i^S$.

The potential itself is therefore known. The optimization problem studied here
is different: rather than evaluating a prescribed target set or ranking
candidate targets, we minimize $|S|$ subject to $h_i^S\ge\tau$ for every
$i\in V$.
\end{remark}

For convenience, set $h^{\varnothing}:=0$.

\begin{definition}[Discounted hitting domination]\label{def:dhd}
Fix $\tau\in(0,1]$. A set $S\subseteq V$ is a
\emph{$(\Lambda,W,\tau)$-dominating set} if
$\min_{i\in V}h_i^S\ge\tau$. Its minimum possible cardinality is the
\emph{discounted hitting domination number}
\[
\delta_{\Lambda,W,\tau}
:=\min\{|S|:S\text{ is a }(\Lambda,W,\tau)\text{-dominating set}\}.
\]
For the simple random walk on an undirected graph $G$ with homogeneous
fidelity $\lambda$, we write $\delta_{\lambda,\tau}(G)$ and use the term
\emph{$(\lambda,\tau)$-dominating set}. For homogeneous fidelity on a general
row-stochastic matrix $W$, write
$\delta_{\lambda,W,\tau}:=\delta_{\lambda I,W,\tau}$.
\end{definition}

\begin{remark}[Standing conventions for simple graphs]\label{rem:conventions}
Whenever $\delta_{\lambda,\tau}(G)$ is used for a simple graph $G$, $W$ denotes
the simple-random-walk matrix of $G$. We therefore assume that $G$ has no
isolated vertices. The only exception is $K_1$, for which we set
$\delta_{\lambda,\tau}(K_1)=1$.
\end{remark}

\begin{remark}[Relation to existing models]\label{rem:relation}
The equilibrium system is a pinned, discounted specialization of linear
stubborn-agent dynamics \cite{friedkin1990,ghaderi2014,hunter2022}. The
optimization criterion is different from mean-shift and threshold-count
objectives because every coordinate must satisfy the same floor. It also
differs from finite-horizon random-walk domination \cite{li2014} and
exponential domination \cite{dankelmann2009}: here all walks contribute through
degree-normalized transition probabilities, and the selected set enters through
the first hitting time rather than through a sum of distance weights.
\end{remark}

\begin{example}[A four-agent verification chain]\label{ex:P4}
Let $W$ be the simple-random-walk matrix of the path
$v_1-v_2-v_3-v_4$, take $S=\{v_1\}$, and set $\lambda_i=\lambda=4/5$.
Then \eqref{eq:affine} becomes
\[
\begin{bmatrix}
x_2(t+1)\\
x_3(t+1)\\
x_4(t+1)
\end{bmatrix}
=
\frac45
\begin{bmatrix}
0&1/2&0\\
1/2&0&1/2\\
0&1&0
\end{bmatrix}
\begin{bmatrix}
x_2(t)\\
x_3(t)\\
x_4(t)
\end{bmatrix}
+
\frac45
\begin{bmatrix}
1/2\\
0\\
0
\end{bmatrix}.
\]
Solving \eqref{eq:equilibrium} gives
$h^{\{v_1\}}=(1,34/65,4/13,16/65)$. Thus one verified source guarantees the
floor $\tau=0.24$ but not $\tau=0.25$, although every vertex is reachable from
$v_1$.
\end{example}

The probabilistic, submodular, and mixed-integer results below retain the full
directed heterogeneous model $(W,\Lambda)$. The distance comparisons and the
exact graph-family results specialize to homogeneous fidelity
$\lambda_i=\lambda$ and the simple random walk.

\section{Structural bounds and recovery of distance domination}\label{sec:structural}

Throughout this section, $G=(V,E)$ is a finite simple graph with no isolated
vertices, $W$ is its simple-random-walk matrix, and the fidelity is homogeneous,
$\lambda_i=\lambda\in(0,1)$. Let $\Delta$ denote the maximum degree of $G$.
If a component contains no source, then $h_v^S=0$ throughout that component.
Accordingly, the lower bound below is stated for vertices with
$\dist(v,S)<\infty$; the upper bound remains valid for
$\dist(v,S)=\infty$ under the convention $\lambda^\infty=0$.

\begin{proposition}[Distance bounds]\label{prop:distance-bounds}
Let $\varnothing\neq S\subseteq V$ and let $v\in V$ with
$d=\dist(v,S)<\infty$. Then
\begin{equation}\label{eq:distance-bounds}
\left(\frac{\lambda}{\Delta}\right)^d
\le h_v^S
\le \lambda^d.
\end{equation}
\end{proposition}

\begin{proof}
Since $T_S\ge d$, \eqref{eq:homogeneous-hitting} gives
$h_v^S=\mathbb{E}_v[\lambda^{T_S}]\le\lambda^d$.

For the lower bound, fix a shortest path
$v=v_0,v_1,\dots,v_d\in S$. None of $v_0,\dots,v_{d-1}$ lies in $S$.
The walk follows this path during its first $d$ steps with probability
$\prod_{r=0}^{d-1}\deg(v_r)^{-1}\ge\Delta^{-d}$, and on this event
$T_S=d$. Hence $h_v^S\ge\lambda^d\Delta^{-d}$.
\end{proof}

\begin{remark}[Sharpness of the distance bounds]\label{rem:sharpness}
Let $d=\dist(v,S)\ge1$. The upper bound is attained exactly when
$d=1$ and $N(v)\subseteq S$. Indeed, in that case $T_S=1$ almost surely.
Conversely, if $d=1$ and some neighbor of $v$ is not a source, then
$\mathbb{P}_v(T_S>1)>0$, so $h_v^S<\lambda$. If $d\ge2$, every neighbor of
$v$ is a non-source. Choosing any neighbor $u$, the walk returns to $v$ in
two steps with probability at least
$(\deg(v)\deg(u))^{-1}>0$; on that event $T_S\ge d+2$, while always
$T_S\ge d$. Thus $h_v^S<\lambda^d$.

Equality in the lower bound occurs only when $d=1$ and $\Delta=1$.
For $d=1$,
$h_v^S=\frac{\lambda}{\deg(v)}
\bigl(|N(v)\cap S|+\sum_{u\in N(v)\setminus S}h_u^S\bigr)$.
Since every vertex in a component containing a source has positive support,
equality with $\lambda/\Delta$ forces $\deg(v)=\Delta=1$. For $d\ge2$,
there is positive probability that the walk deviates from a fixed shortest
path and later hits $S$, giving strictly positive contribution beyond the
single-path term used in the proof of \cref{prop:distance-bounds}.

The upper estimate is nevertheless asymptotically sharp. Fix $d$ and take
layers $A_0=\{v\},A_1,\dots,A_d$ with $|A_i|=M^i$, join consecutive layers
completely, and set $S=A_d$. From each intermediate layer, the walk advances
toward $S$ with probability $M^2/(1+M^2)$, while its first step from $v$ is
forced forward. Hence
$h_v^S\ge\lambda^d(M^2/(1+M^2))^{d-1}\to\lambda^d$ as $M\to\infty$.

The lower estimate is asymptotically sharp at distance one in the
small-fidelity limit. On $K_{\Delta+1}$ with a single source, symmetry gives
$h_v^S=\lambda/(\Delta-\lambda(\Delta-1))$ for every non-source $v$, and
$h_v^S/(\lambda/\Delta)\to1$ as $\lambda\downarrow0$. No fixed-$\lambda$
sharpness of the lower estimate is asserted.
\end{remark}

Let $\gamma_r(G)$ denote the minimum cardinality of a set $S\subseteq V$ such
that $\dist(v,S)\le r$ for every $v\in V$, with $\gamma_0(G)=|V|$. Define
$r_+(\lambda,\tau):=\lfloor\log\tau/\log\lambda\rfloor$ and
$r_-(\lambda,\tau,\Delta):=
\lfloor\log\tau/\log(\lambda/\Delta)\rfloor$.

\begin{corollary}[Distance-domination sandwich]\label{cor:sandwich}
For every such graph,
\[
\gamma_{r_+(\lambda,\tau)}(G)
\le \delta_{\lambda,\tau}(G)
\le \gamma_{r_-(\lambda,\tau,\Delta)}(G).
\]
\end{corollary}

\begin{proof}
If $S$ is a $(\lambda,\tau)$-dominating set, then
$\tau\le h_v^S\le\lambda^{\dist(v,S)}$, so
$\dist(v,S)\le r_+(\lambda,\tau)$ for every $v$. Hence
$\gamma_{r_+}(G)\le|S|$.

Conversely, if $\dist(v,S)\le r_-(\lambda,\tau,\Delta)$ for every $v$, then
\eqref{eq:distance-bounds} gives
$h_v^S\ge(\lambda/\Delta)^{r_-}\ge\tau$. Thus
$\delta_{\lambda,\tau}(G)\le\gamma_{r_-}(G)$.
\end{proof}

When the two distance scales coincide, the bounds become exact.

\begin{theorem}[Exact recovery of distance domination]\label{thm:distance-recovery}
Let $G$ have maximum degree at most $\Delta$, and let $r\ge1$. If
\begin{equation}\label{eq:recovery-window}
\lambda^{r+1}<\tau\le
\left(\frac{\lambda}{\Delta}\right)^r,
\end{equation}
then a set $S$ is $(\lambda,\tau)$-dominating if and only if it is
distance-$r$ dominating. Consequently,
$\delta_{\lambda,\tau}(G)=\gamma_r(G)$. The interval
\eqref{eq:recovery-window} is nonempty exactly when
$\lambda\Delta^r<1$.
\end{theorem}

\begin{proof}
The interval is nonempty precisely when
$\lambda^{r+1}<(\lambda/\Delta)^r$, equivalently
$\lambda\Delta^r<1$.

Suppose first that $S$ is distance-$r$ dominating. For every $v$,
$d:=\dist(v,S)\le r$, and \eqref{eq:distance-bounds} gives
$h_v^S\ge(\lambda/\Delta)^d\ge(\lambda/\Delta)^r\ge\tau$.
Thus $S$ is $(\lambda,\tau)$-dominating.

Conversely, if $S$ is not distance-$r$ dominating, then some vertex $v$
satisfies $\dist(v,S)\ge r+1$, possibly with infinite distance. By
\eqref{eq:distance-bounds},
$h_v^S\le\lambda^{\dist(v,S)}\le\lambda^{r+1}<\tau$.
Hence $S$ is not $(\lambda,\tau)$-dominating. The equivalence implies
$\delta_{\lambda,\tau}(G)=\gamma_r(G)$.
\end{proof}

\begin{corollary}[Recovery of ordinary domination]\label{cor:recovery-one}
If $\lambda^2<\tau\le\lambda/\Delta$, then
$\delta_{\lambda,\tau}(G)=\gamma(G)$ for every graph of maximum degree at
most $\Delta$. This interval is nonempty exactly when $\lambda\Delta<1$.
\end{corollary}

\begin{proof}
Apply \cref{thm:distance-recovery} with $r=1$, for which
$\gamma_1(G)=\gamma(G)$.
\end{proof}

Outside the recovery window, distance alone does not determine the support.

\begin{example}[Equal distance, unequal support]\label{ex:not-distance}
Let $S$ be one endpoint of $P_3$, and let $u$ be the opposite endpoint.
Solving the equilibrium equations gives
$h_u^S=\lambda^2/(2-\lambda^2)$. Now let $S$ be one leaf of $K_{1,3}$ and
let $u$ be another leaf. Again $\dist(u,S)=2$, but
$h_u^S=\lambda^2/(3-2\lambda^2)<\lambda^2/(2-\lambda^2)$.
Thus vertices at the same distance from a source set can receive different
support. In the star, the hub distributes its weight among three neighbors,
reducing the contribution transmitted from the source.
\end{example}

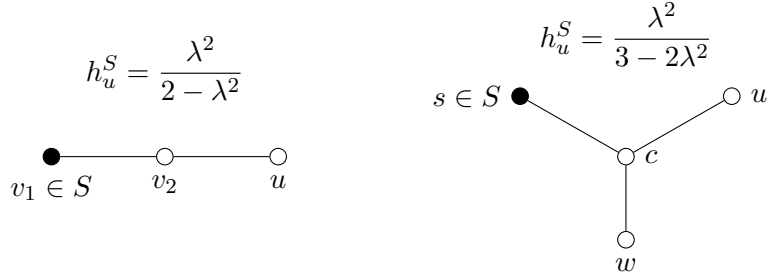
\begin{figure}[t]
\centering
\begin{tikzpicture}[scale=1.0,
  src/.style={circle,draw,fill=black,inner sep=2.2pt},
  agn/.style={circle,draw,inner sep=2.2pt}]
\node[src,label=below:{$v_1\in S$}] (a1) at (0,0) {};
\node[agn,label=below:{$v_2$}] (a2) at (1.5,0) {};
\node[agn,label=below:{$u$}] (a3) at (3.0,0) {};
\draw (a1)--(a2)--(a3);
\node at (1.5,1.15)
  {$h_u^S=\dfrac{\lambda^2}{2-\lambda^2}$};

\begin{scope}[xshift=7.6cm]
\node[agn,label=right:{$c$}] (c) at (0,0) {};
\node[src,label=left:{$s\in S$}] (s) at (-1.4,0.8) {};
\node[agn,label=right:{$u$}] (u) at (1.4,0.8) {};
\node[agn,label=below:{$w$}] (w) at (0,-1.1) {};
\draw (s)--(c)--(u) (c)--(w);
\node at (0,1.6)
  {$h_u^S=\dfrac{\lambda^2}{3-2\lambda^2}$};
\end{scope}
\end{tikzpicture}
\caption{Two graphs with $\dist(u,S)=2$ but different equilibrium support
(\cref{ex:not-distance}). The degree-three hub in $K_{1,3}$ dilutes the
support transmitted from the source.}
\label{fig:dilution}
\end{figure}

\begin{remark}[Finite-time certification]\label{rem:finite-time}
Suppose the non-source coordinates are initialized at zero. By
\cref{thm:convergence},
$0\le h_i^S-x_i(t)\le\lambda_*^t$ for every $i$. Hence, if
$\min_i h_i^S\ge\tau+\varepsilon$ for some $\varepsilon>0$, then
$x_i(t)\ge\tau$ for all $i$ whenever
$t\ge\lceil\log\varepsilon/\log\lambda_*\rceil$.
\end{remark}
\section{Optimization and computational complexity}\label{sec:optimization}

Define the aggregate support
$F(S):=\sum_{i\in V}h_i^S$ for $S\subseteq V$.

\begin{theorem}[Submodularity]\label{thm:submodular}
For every row-stochastic $W$ and every fidelity matrix $\Lambda$ with
$\lambda_*<1$, the set function $S\mapsto h_i^S$ is normalized, monotone,
and submodular for each $i\in V$. Consequently, $F$ is normalized,
monotone, and submodular.
\end{theorem}

\begin{proof}
Fix $i\in V$ and a sample path
$\omega=(X_0,X_1,\dots)$ of the Markov chain started at $i$. For $v\in V$,
set
$a_v(\omega):=\prod_{r<T_{\{v\}}}\lambda_{X_r}$, with
$a_v(\omega)=0$ if $v$ is never visited. Since the running product
$\prod_{r<t}\lambda_{X_r}$ is nonincreasing in $t$ and
$T_S=\min_{v\in S}T_{\{v\}}$, we have
\[
\prod_{r<T_S}\lambda_{X_r}
=\max_{v\in S}a_v(\omega),
\]
with the maximum over the empty set defined to be zero.

For fixed nonnegative values $(a_v)_{v\in V}$, the function
$S\mapsto\max_{v\in S}a_v$ is normalized and monotone. Its marginal gain
from adding $v$ is
$\bigl(a_v-\max_{u\in S}a_u\bigr)_+$, which is nonincreasing as $S$
grows; hence the function is submodular. Taking expectation with respect
to the walk started at $i$ gives the same three properties for
$S\mapsto h_i^S$. Summing over $i$ proves the assertion for $F$.
\end{proof}

This pathwise maximum representation is consistent with the submodularity
mechanisms appearing in related influence and random-walk objectives
\cite{nemhauser1978,kempe2015,li2014,hunter2022}.

\begin{corollary}[Greedy aggregate-support guarantee]\label{cor:greedy}
For a source budget $k$, the standard greedy algorithm returns a set
$S_k$ satisfying
$F(S_k)\ge(1-e^{-1})\max_{|S|\le k}F(S)$.
The algorithm runs in polynomial time, since it makes polynomially many
evaluations of $F$, each obtained by solving the linear system
\eqref{eq:equilibrium}.
\end{corollary}

\begin{proof}
By \cref{thm:submodular}, $F$ is normalized, monotone, and submodular.
The result is the classical cardinality-constrained greedy bound of
\cite{nemhauser1978}.
\end{proof}

The uniform floor also has an exact submodular formulation.

\begin{proposition}[Submodular-cover formulation]\label{prop:submodular-cover}
For $\tau\in(0,1]$, define
$Q_\tau(S):=\sum_{i\in V}\min\{h_i^S,\tau\}$.
Then $Q_\tau$ is normalized, monotone, and submodular. Moreover,
$S$ is a $(\Lambda,W,\tau)$-dominating set if and only if
$Q_\tau(S)=Q_\tau(V)=n\tau$. Hence
$\delta_{\Lambda,W,\tau}
=\min\{|S|:Q_\tau(S)=Q_\tau(V)\}$.
\end{proposition}

\begin{proof}
It suffices to note that truncation at a constant preserves monotonicity
and submodularity for a nonnegative monotone submodular function. For
completeness, let $f$ have these properties and set
$g(S):=\min\{f(S),\tau\}$. For $S\subseteq T$ and $v\notin T$, if
$f(T)\ge\tau$, then the marginal gain at $T$ is zero and the desired
inequality is immediate. If $f(T)<\tau$, then also $f(S)<\tau$, and
\[
g(R\cup\{v\})-g(R)
=
\min\{f(R\cup\{v\})-f(R),\,\tau-f(R)\},
\qquad R\in\{S,T\}.
\]
Both terms inside the minimum are no larger for $R=T$ than for $R=S$:
the first by submodularity and the second by monotonicity. Thus $g$ is
submodular.

Applying this to each $S\mapsto h_i^S$ and summing shows that $Q_\tau$
is normalized, monotone, and submodular. Since each summand is at most
$\tau$, equality $Q_\tau(S)=n\tau$ holds exactly when
$h_i^S\ge\tau$ for every $i$. Finally, $h^V=\one$, so
$Q_\tau(V)=n\tau$.
\end{proof}

Thus discounted hitting domination is an exact submodular-cover problem.
The corresponding greedy rule will be used as a computational baseline in
\cref{sec:experiments}; see \cite{wolsey1982} for the classical analysis of
greedy submodular cover. An exact mixed-integer formulation is also
available.

\begin{proposition}[Exact mixed-integer linear formulation]\label{prop:milp}
For every row-stochastic $W$, every fidelity matrix $\Lambda$ with
$\lambda_*<1$, and every $\tau\in(0,1]$,
$\delta_{\Lambda,W,\tau}$ is the optimal value of
\begin{equation}\label{eq:milp}
\min\left\{
\sum_{i\in V}y_i:
\ \tau\le h_i\le1,\quad
y_i\le h_i,\quad
0\le h_i-\lambda_i\sum_{j\in V}w_{ij}h_j\le y_i,\quad
y_i\in\{0,1\}
\ \text{for all }i\in V
\right\}.
\end{equation}
More precisely, $S\mapsto(\one_S,h^S)$ sends
$(\Lambda,W,\tau)$-dominating sets to feasible solutions of
\eqref{eq:milp}, while $(y,h)\mapsto\{i:y_i=1\}$ sends feasible solutions
to $(\Lambda,W,\tau)$-dominating sets. Both maps preserve cardinality and
objective value.
\end{proposition}
\begin{proof}
Let $S$ be a $(\Lambda,W,\tau)$-dominating set and take
$y=\one_S$, $h=h^S$. Then $\tau\le h_i\le1$ for every $i$. If
$i\notin S$, then $y_i=0$ and the fixed-point equation gives
$h_i-\lambda_i\sum_jw_{ij}h_j=0$. If $i\in S$, then
$y_i=h_i=1$, while
$0<1-\lambda_i\le
h_i-\lambda_i\sum_jw_{ij}h_j\le1$.
Thus $(y,h)$ is feasible and its objective value is $|S|$.

Conversely, let $(y,h)$ be feasible and put
$S:=\{i:y_i=1\}$. If $i\in S$, then
$1=y_i\le h_i\le1$, so $h_i=1$. If $i\notin S$, then $y_i=0$ and
the two-sided constraint gives
$h_i=\lambda_i\sum_jw_{ij}h_j$. The set $S$ cannot be empty: otherwise
$h=\Lambda Wh$, and
$\|\Lambda W\|_\infty\le\lambda_*<1$ would force $h=0$, contradicting
$h_i\ge\tau>0$. Hence $h$ satisfies the pinned fixed-point equations for
$S$, and uniqueness in \cref{thm:convergence} gives $h=h^S$. Therefore
$S$ is $(\Lambda,W,\tau)$-dominating, and the objective value is
$\sum_i y_i=|S|$.
\end{proof}

The formulation \eqref{eq:milp} has $n$ binary and $n$ continuous
variables and $3n$ linear constraints apart from variable bounds. It is
exact, not a relaxation, and is used in \cref{sec:experiments} to certify
the reported optimal source cardinalities.

\begin{problem}[\textsc{Discounted Hitting Domination}]\label{prob:decision}
For fixed rational $\lambda,\tau\in(0,1)$, given a graph $G$ with no
isolated vertices and an integer $k$, decide whether
$\delta_{\lambda,\tau}(G)\le k$.
\end{problem}

\begin{theorem}[Complexity at fixed parameters]\label{thm:hardness}
For $(\lambda,\tau)=(1/4,1/14)$,
{\upshape\textsc{Discounted Hitting Domination}} is NP-complete even on
graphs of maximum degree three. The corresponding minimum-cardinality
optimization problem is APX-complete on the same graph class.
\end{theorem}

\begin{proof}
Let $G$ have maximum degree at most three. Since
$1/16<1/14\le1/12$, we have
$\lambda^2<\tau\le\lambda/3$ at
$(\lambda,\tau)=(1/4,1/14)$. By
\cref{cor:recovery-one}, a set $S\subseteq V(G)$ is
$(1/4,1/14)$-dominating if and only if it is a classical dominating set.
Consequently,
$\delta_{1/4,1/14}(G)=\gamma(G)$, and the feasible solutions coincide
with identical objective values.

Thus the identity map gives an objective-preserving reduction between
Minimum Dominating Set and the present problem on this graph class.
Minimum Dominating Set is NP-complete and APX-complete already on cubic
graphs \cite{alimonti2000}; hence the optimization problem above is
APX-complete and the decision problem is NP-hard on graphs of maximum
degree three.

It remains to verify membership in NP. Given a proposed source set $S$,
the equilibrium is the solution of the rational linear system
\eqref{eq:equilibrium}. Exact Gaussian elimination runs in polynomial time
in the bit model, after which each coordinate can be compared with the
fixed rational threshold $\tau$. Thus feasibility can be verified in
polynomial time.
\end{proof}

\begin{remark}\label{rem:hardness-transfer}
The argument is not specific to ordinary domination. Whenever the parameters
lie in the recovery window of \cref{thm:distance-recovery}, any complexity
or approximation result for minimum distance-$r$ domination on a
bounded-degree graph class transfers directly to discounted hitting
domination, because the two problems have the same feasible sets and
objective values on that class. The spider results of \cref{sec:spiders}
identify a structured family for which the discounted problem is instead
polynomial-time solvable.
\end{remark}
\section{Transfer formulas and admissible scales}\label{sec:attenuation}

Throughout this section, $W$ is the simple-random-walk matrix and the fidelity
is homogeneous, $\lambda_i=\lambda\in(0,1)$. Set
$\mu:=\arcosh(1/\lambda)>0$, so that $\lambda=\sech\mu$.

Because sources are pinned, the equilibrium equations decouple across the
source-free path segments determined by $S$. Two boundary-value problems are
needed for spiders: a finite segment terminating at a source, and a terminal
tail ending at a degree-one vertex. We first record the transfer sequences
common to both. Define
\begin{equation}\label{eq:pq-recurrence}
p_0=0,\qquad p_1=1,\qquad
q_0=1,\qquad q_1=\frac1\lambda,
\qquad
y_{k+1}=\frac{2}{\lambda}y_k-y_{k-1}
\quad (y\in\{p,q\}).
\end{equation}

\begin{lemma}[Transfer sequences]\label{lem:transfer-sequences}
For every $k\ge0$,
$p_k=\sinh(k\mu)/\sinh\mu$ and $q_k=\cosh(k\mu)$. Both sequences are
strictly increasing, with $p_k>0$ for $k\ge1$ and $q_k\ge1$. Moreover, for
$1\le j\le m-1$,
\begin{equation}\label{eq:pq-identity}
\frac{p_{m-j}+p_j}{p_m}
=
\frac{\cosh((j-m/2)\mu)}{\cosh(m\mu/2)}.
\end{equation}
If $\lambda\in\Q$, then every $p_k$ and $q_k$ is rational and can be computed
exactly from \eqref{eq:pq-recurrence}.
\end{lemma}

\begin{proof}
Since $1/\lambda=\cosh\mu$, the recurrence in
\eqref{eq:pq-recurrence} is
$y_{k+1}-2\cosh(\mu)y_k+y_{k-1}=0$. Its solutions are linear combinations
of $\cosh(k\mu)$ and $\sinh(k\mu)$, and the initial conditions give the
stated formulas for $p_k$ and $q_k$. Their monotonicity follows from that of
$\sinh$ and $\cosh$ on $[0,\infty)$.

For \eqref{eq:pq-identity}, use
$\sinh((m-j)\mu)+\sinh(j\mu)
=2\sinh(m\mu/2)\cosh((j-m/2)\mu)$ and
$\sinh(m\mu)=2\sinh(m\mu/2)\cosh(m\mu/2)$.
Rationality follows directly from the recurrence.
\end{proof}

\begin{remark}[Exact integer arithmetic]\label{rem:integer}
Suppose $\lambda=c/e$ in lowest terms, with $0<c<e$. Define
$\widetilde p_0:=0$,
$\widetilde p_k:=c^{k-1}p_k$ for $k\ge1$, and
$\widetilde q_k:=c^kq_k$ for $k\ge0$. Then both scaled sequences satisfy
$\widetilde y_{k+1}=2e\,\widetilde y_k-c^2\widetilde y_{k-1}$, with
$\widetilde p_1=1$, $\widetilde q_0=1$, and $\widetilde q_1=e$.
Thus $\widetilde p_k$ and $\widetilde q_k$ are positive integers for
$k\ge1$, of bit-size $O(k\log e)$.

If $\tau=s/t$ is rational in lowest terms, all transfer quantities used below,
including \eqref{eq:theta}, \eqref{eq:LB}, and the leg data in
\eqref{eq:Fdata}--\eqref{eq:Ndata}, are rational with polynomial bit-size.
Their comparisons may therefore be performed exactly by cross-multiplication,
without numerical rounding.
\end{remark}

\begin{lemma}[Two-boundary segment]\label{lem:two-boundary}
Let $m\ge1$ and $x\in[0,1]$. Suppose $z_0,\dots,z_m$ satisfy
$z_0=x$, $z_m=1$, and
$z_j=\frac{\lambda}{2}(z_{j-1}+z_{j+1})$ for $1\le j\le m-1$.
Then the solution is unique and
\begin{equation}\label{eq:two-boundary}
z_j(x;m)=\frac{p_{m-j}x+p_j}{p_m},
\qquad 0\le j\le m.
\end{equation}
Each $z_j(x;m)$ is nondecreasing in $x$, and
$z_1(x;m)=A_mx+C_m$, where
$A_m:=p_{m-1}/p_m\in[0,1)$ and $C_m:=1/p_m$.
In particular, $A_1=0$ and $A_m>0$ for $m\ge2$.

For $m\ge2$, all interior values satisfy $z_j\ge\tau$ if and only if
$x\ge\vartheta_m$, where
\begin{equation}\label{eq:theta}
\vartheta_m
:=
\max\left\{
0,\,
\max_{1\le j\le m-1}
\frac{\tau p_m-p_j}{p_{m-j}}
\right\},
\end{equation}
and we set $\vartheta_1:=0$.
\end{lemma}

\begin{proof}
Formula \eqref{eq:two-boundary} satisfies the recurrence by
\cref{lem:transfer-sequences} and has the prescribed boundary values because
$p_0=0$. Uniqueness follows from strict diagonal dominance of the interior
linear system: its diagonal entries are one and its off-diagonal row sums are
at most $\lambda<1$.

Monotonicity in $x$ follows from $p_{m-j}\ge0$, while
$0\le A_m<1$ follows from $p_0=0$ and
$p_{m-1}<p_m$ for $m\ge2$. Finally,
$z_j(x;m)\ge\tau$ is equivalent to
$p_{m-j}x\ge\tau p_m-p_j$. Since $p_{m-j}>0$ for
$1\le j\le m-1$, imposing these inequalities simultaneously gives exactly
$x\ge\vartheta_m$.
\end{proof}

\begin{lemma}[Reflecting tail]\label{lem:tail}
Let $b\ge1$ and $x\in[0,1]$. Suppose $z_0,\dots,z_b$ satisfy
$z_0=x$,
$z_j=\frac{\lambda}{2}(z_{j-1}+z_{j+1})$ for $1\le j\le b-1$, and
$z_b=\lambda z_{b-1}$. Then the solution is unique and
\begin{equation}\label{eq:tail-solution}
z_j=x\,\frac{q_{b-j}}{q_b},
\qquad 0\le j\le b.
\end{equation}
The sequence is nonincreasing in $j$, and
$z_1(x;b)=A_b^0x$ with $A_b^0:=q_{b-1}/q_b\in(0,1)$.
Moreover,
$\min_{1\le j\le b}z_j\ge\tau$ if and only if $x\ge\tau q_b$.
\end{lemma}

\begin{proof}
The formula in \eqref{eq:tail-solution} satisfies the interior recurrence by
\cref{lem:transfer-sequences}. At the endpoint,
$\lambda z_{b-1}
=\lambda xq_1/q_b
=x/q_b
=z_b$, since $q_1=1/\lambda$.
Uniqueness again follows from strict diagonal dominance; the final row has
off-diagonal coefficient $\lambda<1$.

Because $(q_k)$ is increasing, $z_j$ is nonincreasing in $j$, and its minimum
is $z_b=x/q_b$. Hence $\min_j z_j\ge\tau$ exactly when
$x\ge\tau q_b$.
\end{proof}

Setting $x=1$ yields the two extremal profiles used below. For a gap of
$m$ edges between two sources, \eqref{eq:pq-identity} shows that the minimum
interior value occurs at the central index $j=\lfloor m/2\rfloor$, so
\[
\psi_m
:=
\min_{1\le j\le m-1}z_j(1;m)
=
\begin{cases}
\sech(m\mu/2), & m \text{ even},\\[1mm]
\dfrac{\cosh(\mu/2)}{\cosh(m\mu/2)}, & m \text{ odd},
\end{cases}
\]
with $\psi_1:=1$. For a source-terminated tail of $\ell$ edges, the minimum
value is $\eta_\ell:=1/q_\ell=\sech(\ell\mu)$.

\begin{lemma}[Admissible gap and tail lengths]\label{lem:monotone-psi}
The sequences $(\psi_m)_{m\ge1}$ and $(\eta_\ell)_{\ell\ge0}$ are strictly
decreasing to zero. Consequently, the gap lengths and tail lengths whose
minimum support is at least $\tau$ are respectively
$\{1,\dots,L\}$ and $\{0,\dots,B\}$, where
\begin{equation}\label{eq:LB}
L:=\max\{m\ge1:\psi_m\ge\tau\},
\qquad
B:=\max\{\ell\ge0:\eta_\ell\ge\tau\}.
\end{equation}
Both $L$ and $B$ are finite and well defined.
\end{lemma}

\begin{proof}
Since $\eta_\ell=\sech(\ell\mu)$ and $\mu>0$, the sequence
$(\eta_\ell)$ is strictly decreasing to zero.

For $\psi$, write
$\psi_{2r}=\sech(r\mu)$,
$\psi_{2r-1}=\cosh(\mu/2)/\cosh((r-\tfrac12)\mu)$, and
$\psi_{2r+1}=\cosh(\mu/2)/\cosh((r+\tfrac12)\mu)$.
Using
\[
\cosh(\mu/2)\cosh(r\mu)
=
\frac12\left[
\cosh((r+\tfrac12)\mu)
+
\cosh((r-\tfrac12)\mu)
\right]
\]
and the strict increase of $\cosh$ on $[0,\infty)$ gives
$\psi_{2r-1}>\psi_{2r}>\psi_{2r+1}$ for every $r\ge1$.
Thus $(\psi_m)$ is strictly decreasing to zero. Since
$\psi_1=\eta_0=1\ge\tau$, the corresponding admissible lengths are the stated
finite initial intervals.
\end{proof}
\section{Spiders: exact formula and a polynomial algorithm}\label{sec:spiders}

A \emph{spider} is a tree with at most one vertex of degree greater than two.
We exclude $K_1$, for which $\delta_{\lambda,\tau}(K_1)=1$ by
\cref{rem:conventions}. Every other spider can be written as
$T=\Sp(\ell_1,\dots,\ell_d)$ by choosing a vertex $c$ such that each
component of $T-c$ is a path. The $d\ge1$ components are the legs, with
lengths $\ell_1,\dots,\ell_d\ge1$, and $u_{i,j}$ denotes the vertex at
distance $j$ from $c$ on leg $i$. When $d\ge3$, the center $c$ is the unique
branching vertex; for paths the choice of $c$ need not be unique, and the
statements below hold for any such representation. Stars are the case
$\ell_1=\cdots=\ell_d=1$. Throughout this section, $(\lambda,\tau)$ is fixed,
and $B$ and $L$ are the admissible tail and gap lengths defined in
\eqref{eq:LB}.

We first record the cost of completing a leg once a source has been fixed.

\begin{lemma}[Continuation cost]\label{lem:continuation}
Consider a path segment of $m\ge0$ edges whose left endpoint is a source and
whose right endpoint has degree one. The minimum number of additional sources
needed to keep every vertex at support at least $\tau$ is
\begin{equation}\label{eq:Rdef}
R(m):=\left\lceil\frac{\pospart{m-B}}{L}\right\rceil.
\end{equation}
\end{lemma}

\begin{proof}
Suppose $r$ additional sources are used. Let
$g_1,\dots,g_r\ge1$ be the successive source-to-source gaps and let $b\ge0$
be the terminal tail, so that
$g_1+\cdots+g_r+b=m$. By
\cref{lem:two-boundary,lem:tail}, feasibility is equivalent to
$g_j\le L$ for every $j$ and $b\le B$. Hence $m\le rL+B$, which gives
$r\ge\lceil(m-B)_+/L\rceil$.

Conversely, let $r$ be the right-hand side of \eqref{eq:Rdef}. If $m\le B$,
then $r=0$ and the source-free tail is feasible. If $m>B$, then
$r\le m-B\le rL$, so $m-B$ can be written as a sum of $r$ positive integers,
each at most $L$. Use these integers as the source-to-source gaps and take the
terminal tail to have length $B$. This realizes a feasible placement with
exactly $r$ additional sources.
\end{proof}

\subsection{Center coupling and leg types}

We first consider solutions containing the center.

\begin{proposition}[Cost with a source at the center]\label{prop:center-source}
Among $(\lambda,\tau)$-dominating sets of
$T=\Sp(\ell_1,\dots,\ell_d)$ that contain $c$, the minimum cardinality is
\begin{equation}\label{eq:Dc}
D_c=1+\sum_{i=1}^dR(\ell_i).
\end{equation}
\end{proposition}

\begin{proof}
Once $c$ is pinned at one, the legs decouple. Leg $i$ is a source-terminated
path segment of length $\ell_i$, so \cref{lem:continuation} gives the minimum
additional cost $R(\ell_i)$. Summing over the legs and adding the center yields
\eqref{eq:Dc}.
\end{proof}

Now suppose $c\notin S$ and write $x=h_c^S$. Each leg presents one of two
transfer types to the center.

If the first source on leg $i$ occurs at distance $a$ from $c$, then
$1\le a\le\min\{L,\ell_i\}$. We call this the \emph{first-source type}
$F_a$. By \cref{lem:two-boundary}, the source-free prefix between $c$ and
$u_{i,a}$ has transfer data
\begin{equation}\label{eq:Fdata}
\bigl(A(F_a),C(F_a),\vartheta(F_a)\bigr)
=
\left(\frac{p_{a-1}}{p_a},\frac1{p_a},\vartheta_a\right),
\qquad
\kappa_i(F_a)=1+R(\ell_i-a).
\end{equation}
Here $\kappa_i(F_a)$ is the minimum number of sources contributed by leg $i$:
the first source at $u_{i,a}$ together with an optimal continuation of the
remaining suffix.

If leg $i$ contains no source, then necessarily $\ell_i\le B$. Writing
$b=\ell_i$, we call this the \emph{source-free type} $N_b$. By
\cref{lem:tail},
\begin{equation}\label{eq:Ndata}
\bigl(A(N_b),C(N_b),\vartheta(N_b)\bigr)
=
\left(\frac{q_{b-1}}{q_b},0,\tau q_b\right),
\qquad
\kappa_i(N_b)=0.
\end{equation}

Let $\mathcal O_i$ be the set of types available to leg $i$. For a type vector
$\boldsymbol t=(t_1,\dots,t_d)\in\prod_{i=1}^d\mathcal O_i$, define
\begin{equation}\label{eq:sumdata}
\begin{aligned}
A(\boldsymbol t)&:=\sum_{i=1}^dA(t_i),&
C(\boldsymbol t)&:=\sum_{i=1}^dC(t_i),\\
\Theta(\boldsymbol t)&:=
\max\bigl(\{\tau\}\cup\{\vartheta(t_i):1\le i\le d\}\bigr),&
K(\boldsymbol t)&:=\sum_{i=1}^d\kappa_i(t_i).
\end{aligned}
\end{equation}
The corresponding center value is
\begin{equation}\label{eq:center-value}
X(\boldsymbol t)
:=
\frac{\lambda C(\boldsymbol t)}
     {d-\lambda A(\boldsymbol t)}.
\end{equation}
The denominator is positive because $A(t_i)<1$ for every available type, so
$\lambda A(\boldsymbol t)<d$.

\begin{theorem}[Exact spider formula]\label{thm:spider-exact}
For every spider $T=\Sp(\ell_1,\dots,\ell_d)$,
\begin{equation}\label{eq:spider-formula}
\delta_{\lambda,\tau}(T)=\min\{D_c,D_{\bar c}\},
\end{equation}
where $D_c$ is given by \eqref{eq:Dc} and
\begin{equation}\label{eq:noncenter-cost}
D_{\bar c}
:=
\min\left\{
K(\boldsymbol t):
\boldsymbol t\in\prod_{i=1}^d\mathcal O_i,\
X(\boldsymbol t)\ge\Theta(\boldsymbol t)
\right\},
\end{equation}
with $D_{\bar c}:=\infty$ if no feasible type vector exists.
\end{theorem}

\begin{proof}
The case $c\in S$ is \cref{prop:center-source}. Assume therefore that
$c\notin S$, and put $x=h_c^S$.

First consider any $(\lambda,\tau)$-dominating set $S$ avoiding $c$.
Fix a leg $i$. If the leg contains no source, then
\cref{lem:tail} gives leaf value $x/q_{\ell_i}\le1/q_{\ell_i}=\eta_{\ell_i}$.
Feasibility therefore implies $\ell_i\le B$, so the leg has type
$N_{\ell_i}$.

If the leg contains a source, let $a$ be the distance from $c$ to its first
source. The prefix between $c$ and that source has boundary values $x$ and
$1$. Since $x\le1$, its support values are bounded above by those obtained
with both boundary values equal to one. Feasibility therefore requires
$\psi_a\ge\tau$, hence $a\le L$, and the leg has type $F_a$. Once
$u_{i,a}$ is pinned, the remaining suffix is independent of the rest of the
tree, so \cref{lem:continuation} requires at least $R(\ell_i-a)$ additional
sources. Thus the source set determines a type vector $\boldsymbol t$ and
satisfies $|S|\ge K(\boldsymbol t)$.

The center equation is
\begin{equation}\label{eq:center-eq}
x
=
\frac{\lambda}{d}
\sum_{i=1}^d\bigl(A(t_i)x+C(t_i)\bigr),
\end{equation}
whose unique solution is $x=X(\boldsymbol t)$. The center itself must satisfy
$x\ge\tau$, and each leg must satisfy the threshold encoded by its type.
Together these conditions are exactly
$X(\boldsymbol t)\ge\Theta(\boldsymbol t)$. Hence every feasible source set
avoiding $c$ has cardinality at least $D_{\bar c}$.

Conversely, let $\boldsymbol t$ be feasible in
\eqref{eq:noncenter-cost}. For each leg of type $F_a$, place a source at
$u_{i,a}$ and complete its suffix with an optimal placement from
\cref{lem:continuation}; leave each leg of type $N_b$ source-free. The
resulting set $S$ has cardinality $K(\boldsymbol t)$.

Set $\hat x:=X(\boldsymbol t)$. Extend $\hat x$ along each first-source prefix
using \eqref{eq:two-boundary}, along each source-free leg using
\eqref{eq:tail-solution}, and along each suffix using its pinned equilibrium.
The resulting vector satisfies all equilibrium equations: the path equations
hold by \cref{lem:two-boundary,lem:tail}, the selected sources are pinned, and
the center equation is \eqref{eq:center-eq}. By uniqueness in
\cref{thm:convergence}, this vector is the equilibrium $h^S$.

Since $X(\boldsymbol t)\ge\Theta(\boldsymbol t)$, the center, every prefix,
and every source-free leg satisfy the support floor; the suffixes do so by
\cref{lem:continuation}. Hence $S$ is $(\lambda,\tau)$-dominating.
Minimizing over feasible type vectors gives $D_{\bar c}$, and comparison with
the center-selected case proves \eqref{eq:spider-formula}.
\end{proof}

\begin{remark}[The center need not be selected]\label{rem:center-not-always}
Both terms in \eqref{eq:spider-formula} are essential. In
\cref{ex:spider222}, the unique branching vertex belongs to no minimum
$(\lambda,\tau)$-dominating set. Thus an optimal placement need not contain
the unique maximum-degree vertex, even on a spider.
\end{remark}

\subsection{Dynamic programming over type counts}

The minimization in \eqref{eq:noncenter-cost} ranges over
$\prod_i|\mathcal O_i|$ assignments and is therefore exponential in the
number of legs if evaluated directly. For fixed $(\lambda,\tau)$, however,
the type alphabet
\[
\mathcal T:=\{F_1,\dots,F_L,N_1,\dots,N_B\},
\qquad
M:=|\mathcal T|=L+B,
\]
has constant size.

For a type vector $\boldsymbol t$, let
$\boldsymbol r=(r_t)_{t\in\mathcal T}\in\N^M$ be its count vector, where
$r_t$ is the number of legs assigned type $t$. The quantities
$A(\boldsymbol t)$, $C(\boldsymbol t)$, $\Theta(\boldsymbol t)$, and therefore
the feasibility condition $X(\boldsymbol t)\ge\Theta(\boldsymbol t)$, depend
only on $\boldsymbol r$. The assignment cost does not: $\kappa_i(t)$ depends
on the length of leg $i$. This separation leads to a dynamic program.

For $0\le i\le d$, let $\mathrm{DP}_i(\boldsymbol r)$ be the minimum cost
$\sum_{j=1}^i\kappa_j(t_j)$ over assignments
$t_j\in\mathcal O_j$ of the first $i$ legs having count vector
$\boldsymbol r$, with value $\infty$ if no such assignment exists. Set
$\mathrm{DP}_0(\boldsymbol0)=0$. Then
\begin{equation}\label{eq:DP}
\mathrm{DP}_i(\boldsymbol r)
=
\min_{\substack{t\in\mathcal O_i\\ r_t\ge1}}
\left\{
\mathrm{DP}_{i-1}(\boldsymbol r-\boldsymbol e_t)
+\kappa_i(t)
\right\}.
\end{equation}

\begin{theorem}[Exact polynomial-time algorithm on spiders]
\label{thm:spider-algorithm}
Fix rational $\lambda\in(0,1)$ and $\tau\in(0,1]$, and let
$M:=B(\lambda,\tau)+L(\lambda,\tau)$. Given an $n$-vertex spider,
$\delta_{\lambda,\tau}$ and a minimum source set attaining it can be computed
exactly using $O(M^2n^M)$ arithmetic operations on integers of bit-size
$O(M^2\log e+\log t+\log n)$, where $\lambda=c/e$ and $\tau=s/t$ are in
lowest terms. Thus, for every fixed rational pair $(\lambda,\tau)$, discounted
hitting domination is solvable in polynomial time on spiders.

If $\lambda$ and $\tau$ are part of the input, the running time is
$n^M\operatorname{poly}(M,\log e,\log t,\log n)$; hence the algorithm is XP
with parameter $M$.
\end{theorem}

\begin{proof}
By \eqref{eq:sumdata}, the quantities
$A(\boldsymbol t)$, $C(\boldsymbol t)$, and $\Theta(\boldsymbol t)$ depend
only on the count vector of $\boldsymbol t$. Hence all assignments having the
same count vector are either all feasible or all infeasible. It follows that
$D_{\bar c}$ is the minimum of $\mathrm{DP}_d(\boldsymbol r)$ over count
vectors $\boldsymbol r$ satisfying $\sum_t r_t=d$ and
$X(\boldsymbol r)\ge\Theta(\boldsymbol r)$. Recurrence \eqref{eq:DP} follows
by conditioning on the type assigned to leg $i$, and its correctness follows
inductively. Predecessor pointers recover an optimal type assignment; the
corresponding source set is then reconstructed using
\cref{lem:continuation}.

For the state count, a vector reachable after $i$ legs has nonnegative
coordinates summing to $i$. Hence there are at most
$\binom{i+M-1}{M-1}=O(i^{M-1})$ states in layer $i$, and therefore
$O(n^M)$ states over all layers. Each state has at most $M$ incoming
transitions. Maintaining the aggregate transfer data requires at most
$O(M)$ arithmetic work per transition, giving
$O(M^2n^M)$ arithmetic operations.

It remains to control the bit-size. The values $B$ and $L$ can be determined
without evaluating transcendental functions. By
\cref{lem:monotone-psi}, admissible tails and gaps form initial intervals, so
one tests the rational inequalities $1/q_b\ge\tau$ and
$\min_{1\le j\le m-1}(p_{m-j}+p_j)/p_m\ge\tau$ until the first failure in
each sequence. By \cref{rem:integer}, these tests reduce to exact integer
comparisons.

Only $p_k$ and $q_k$ with $k\le M+1$ are needed. Define
\[
\Pi
:=
t\,c^M
\prod_{a=1}^{L}\widetilde p_a
\prod_{b=1}^{B}\widetilde q_b.
\]
Then
$\log\Pi=O(M^2\log e+\log t)$, and the denominators of all transfer data in
\eqref{eq:Fdata}--\eqref{eq:Ndata} divide $\Pi$. Hence, for every count
vector, the scaled quantities
$\widehat A:=\Pi A$,
$\widehat C:=\Pi C$, and
$\widehat\Theta:=\Pi\Theta$
are integers of bit-size
$O(M^2\log e+\log t+\log n)$.

Since $d-\lambda A>0$, the feasibility condition $X\ge\Theta$ is equivalent
to the integer inequality
\[
c\,\widehat C\,\Pi
\ge
\widehat\Theta\,
\bigl(e\,d\,\Pi-c\,\widehat A\bigr).
\]
Both sides have bit-size
$O(M^2\log e+\log t+\log n)$. Thus each arithmetic operation has polynomial
bit complexity in the stated parameters, which proves the running-time
bounds.
\end{proof}

\begin{corollary}[Exact one-source criterion]\label{cor:one-source}
Let $T=\Sp(\ell_1,\dots,\ell_d)$. Then
$\delta_{\lambda,\tau}(T)=1$ if and only if one of the following holds:

\begin{enumerate}[label=\textup{(\roman*)}]
\item $\ell_i\le B$ for every $i$, in which case $\{c\}$ is
$(\lambda,\tau)$-dominating;

\item there exist a leg $i$ and an integer $a$ satisfying
$\max\{1,\ell_i-B\}\le a\le\min\{L,\ell_i\}$ such that
$\ell_j\le B$ for every $j\ne i$ and
\begin{equation}\label{eq:one-source-off-center}
\frac{\lambda/p_a}
{d-\lambda\left(
p_{a-1}/p_a+\sum_{j\ne i}q_{\ell_j-1}/q_{\ell_j}
\right)}
\ge
\max\left\{
\tau,\vartheta_a,\max_{j\ne i}\tau q_{\ell_j}
\right\}.
\end{equation}
In this case $\{u_{i,a}\}$ is $(\lambda,\tau)$-dominating.
\end{enumerate}

An empty sum and an empty maximum are understood to be zero. Thus, when
$d=1$, \eqref{eq:one-source-off-center} becomes
$(\lambda/p_a)/(1-\lambda p_{a-1}/p_a)
\ge\max\{\tau,\vartheta_a\}$.
\end{corollary}

\begin{proof}
The first alternative is exactly $D_c=1$ in \eqref{eq:Dc}, because
$R(\ell_i)=0$ if and only if $\ell_i\le B$.

Now suppose a one-source solution avoids $c$. Exactly one leg, say $i$, has
type $F_a$, while every other leg has type $N_{\ell_j}$. Its suffix cost
must vanish, so $\ell_i-a\le B$; together with availability of $F_a$, this
is equivalent to
$\max\{1,\ell_i-B\}\le a\le\min\{L,\ell_i\}$. The remaining legs must satisfy
$\ell_j\le B$. Substituting these types into
\eqref{eq:center-value} and \eqref{eq:sumdata} gives precisely
\eqref{eq:one-source-off-center}. The result now follows from
\cref{thm:spider-exact}.
\end{proof}

\begin{corollary}[Stars]\label{cor:star-from-spider}
For $m\ge1$,
$\delta_{\lambda,\tau}(K_{1,m})=1$ if $\tau\le\lambda$, and
$\delta_{\lambda,\tau}(K_{1,m})=m+1$ if $\tau>\lambda$.
\end{corollary}

\begin{proof}
Write $K_{1,m}=\Sp(1,\dots,1)$. If $\tau\le\lambda=\eta_1$, then
$B\ge1$, so every leg has length at most $B$ and the center alone is feasible
by \cref{cor:one-source}.

If $\tau>\lambda$, then $B=0$. No source-free leg type is available. If the
center is not selected, every leg must have type $F_1$, which gives
$A(\boldsymbol t)=0$, $C(\boldsymbol t)=m$, and
$X(\boldsymbol t)=\lambda<\tau$; hence $D_{\bar c}=\infty$.
If the center is selected, then $R(1)=1$, so
$D_c=1+m$. Therefore every vertex must be a source.
\end{proof}

\begin{example}[The branching vertex can be absent from every optimum]
\label{ex:spider222}
Take $\lambda=4/5$ and $\tau=1/2$. Then $B=1$ and $L=3$. For
$T=\Sp(2,2,2)$, selecting the center costs
$D_c=1+3R(2)=4$.

Since every leg has length $2>B$, no source-free type is available, so any
solution avoiding $c$ uses at least one source on each leg and therefore has
size at least three. Assigning type $F_1$ to all three legs gives
$A(\boldsymbol t)=0$, $C(\boldsymbol t)=3$, and
$X(\boldsymbol t)=4/5\ge1/2=\Theta(\boldsymbol t)$. Hence
\[
\delta_{4/5,1/2}\bigl(\Sp(2,2,2)\bigr)=3,
\]
attained by
$S=\{u_{1,1},u_{2,1},u_{3,1}\}$.

No three-element solution containing $c$ is feasible: with only two further
sources, at least one leg has no additional source, leaving a tail of length
two, whereas $R(2)=1$. Thus the unique branching vertex belongs to no minimum
$(4/5,1/2)$-dominating set.
\end{example}

\begin{figure}[t]
\centering
\begin{tikzpicture}[scale=1.1,
 src/.style={circle,draw,fill=black,inner sep=2.3pt},
 agn/.style={circle,draw,inner sep=2.3pt}]
\node[agn,label=below:{$c$}] (c) at (0,0) {};
\foreach \ang/\i in {90/1,210/2,330/3}{
  \node[src] (s\i) at (\ang:1.45) {};
  \node[agn] (l\i) at (\ang:2.7) {};
  \draw (c)--(s\i)--(l\i);
}
\node at (0,-2.0)
{$S=\{u_{1,1},u_{2,1},u_{3,1}\}$,\quad
$h_c^S=4/5$,\quad $h^S_{u_{i,2}}=4/5$};
\end{tikzpicture}
\caption{A minimum $(4/5,1/2)$-dominating set of
$\Sp(2,2,2)$ (\cref{ex:spider222}). Filled vertices are sources; the
branching vertex is not selected.}
\label{fig:spider222}
\end{figure}
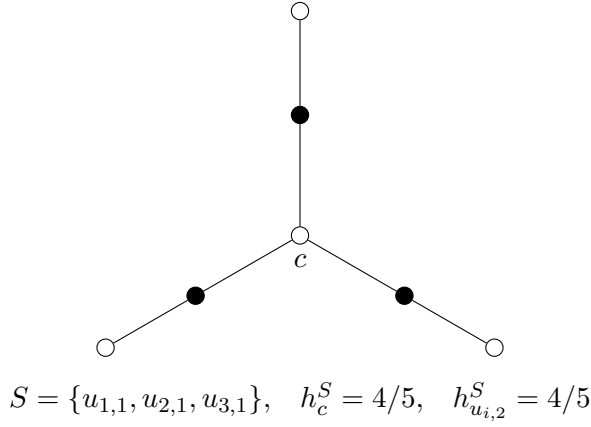

\section{Complete graphs}\label{sec:dense-sparse}

The star formula of \cref{cor:star-from-spider} provides a sparse benchmark:
for $K_{1,m}$, one source suffices when $\tau\le\lambda$, whereas
$\tau>\lambda$ forces every vertex to be selected. Complete graphs exhibit a
different behavior below this one-step threshold: each additional source
raises the common support of all remaining vertices.

\begin{proposition}[Complete graphs]\label{prop:complete}
Let $n\ge2$. For $0<\tau<1$,
\[
\delta_{\lambda,\tau}(K_n)
=
\min\left\{
n,\,
\left\lceil
\frac{\tau(n-1)(1-\lambda)}
     {\lambda(1-\tau)}
\right\rceil
\right\},
\]
and $\delta_{\lambda,1}(K_n)=n$.
\end{proposition}

\begin{proof}
Suppose $q<n$ vertices are sources. By symmetry, every non-source has the
same equilibrium support $y_q$. Since each vertex has degree $n-1$,
\[
y_q
=
\frac{\lambda}{n-1}
\bigl(q+(n-q-1)y_q\bigr),
\]
and therefore
$y_q=\lambda q/[\,n-1-\lambda(n-q-1)\,]$.
The condition $y_q\ge\tau$ is equivalent to
$q\ge\tau(n-1)(1-\lambda)/[\lambda(1-\tau)]$.
Taking the least admissible integer $q$, with $q=n$ when no proper source set
is feasible, gives the formula. If $\tau=1$, every non-source has support
strictly below one, so all $n$ vertices must be sources.
\end{proof}

\begin{remark}[Threshold behavior on complete graphs]\label{rem:complete-threshold}
For $0<\tau<\lambda$, the minimum source fraction satisfies
\[
\frac{\delta_{\lambda,\tau}(K_n)}{n}
\longrightarrow
\frac{\tau(1-\lambda)}
     {\lambda(1-\tau)}
\qquad\text{as }n\to\infty.
\]
At the boundary $\tau=\lambda$ one has
$\delta_{\lambda,\lambda}(K_n)=n-1$, while
$\delta_{\lambda,\tau}(K_n)=n$ for $\tau>\lambda$.
Thus complete graphs admit a gradual increase in the required number of
sources below the one-step fidelity $\lambda$. In contrast,
$K_{1,m}$ remains feasible with a single central source throughout
$\tau\le\lambda$ and then jumps to full selection when $\tau>\lambda$.
\end{remark}

\section{Source placement versus centrality and domination}
\label{sec:experiments}

We next compare discounted-floor placement with standard centrality rankings
and ordinary domination on one real and one synthetic network. The examples
are chosen to test two distinctions suggested by the theory: whether
high-centrality vertices yield efficient source placements, and whether
classical domination is a reliable proxy for the discounted support floor.

For the Zachary karate-club network, $W$ is the row normalization
$w_{ij}=\omega_{ij}/\sum_k\omega_{ik}$ of the integer interaction weights
$\omega_{ij}$ supplied with the data \cite{zachary1977}. For the unweighted
Barab\'asi--Albert graph $\mathrm{BA}(120,2)$ with seed $7$, $W$ is the
simple-random-walk matrix. Equilibria were computed by direct solution of
\eqref{eq:equilibrium}. Network operations used NetworkX
\cite{hagberg2008}, and the mixed-integer programs were solved with HiGHS
through the \texttt{milp} routine of SciPy \cite{virtanen2020}. All four
mixed-integer runs ($\delta$ and $\gamma$ on both networks) terminated with
optimal status, zero reported MIP gap, and dual bound equal to the reported
optimum. The reproducibility code and automated verification workflow are
described in the Data and code availability statement.

We compare four placement procedures. The first, denoted HD, is the greedy
rule associated with the submodular-cover function $Q_\tau$ of
\cref{prop:submodular-cover}: at each step it adds a vertex of maximum
marginal gain, with ties broken by smallest vertex index. The second and
third add vertices in decreasing order of degree and closeness centrality
\cite{freeman1979}, respectively, again breaking ties by index. The fourth is
the standard greedy dominating-set rule, which repeatedly selects a vertex
whose closed neighborhood contains the largest number of currently
undominated vertices.

The exact discounted optimum is computed from \eqref{eq:milp}. Every source
set returned by the solver was independently re-evaluated using
\eqref{eq:equilibrium} to verify the support floor (worst-vertex supports
$0.595$ on Karate and $0.303$ on BA, far above solver tolerance). The
domination number $\gamma$ was computed from the standard binary formulation
$\min\{\sum_i y_i:\sum_{j\in N[i]}y_j\ge1\text{ for all }i,\,
y\in\{0,1\}^V\}$.
For Karate, exhaustive enumeration provides an additional check:
no four-vertex source set reaches $\tau=0.55$, and the largest worst-vertex
support among all four-vertex sets is $0.524$. Since HD attains five sources,
writing $W_{\mathrm K}$ for the normalized Zachary weight matrix,
$\delta_{0.85,W_{\mathrm K},0.55}=5$. On the BA graph,
$\delta_{0.85,0.30}=7$.

\begin{table}[t]
\centering
\small
\begin{tabular}{@{}l c cccc c ccc@{}}
\toprule
& & \multicolumn{4}{c}{Sources to meet floor}
& & \multicolumn{3}{c}{Ordinary domination}\\
\cmidrule(lr){3-6}\cmidrule(lr){8-10}
Network
& $(\lambda,\tau)$
& $\delta$
& HD
& degree
& closeness
&
& $\gamma$
& $|D_{\mathrm g}|$
& $\min_i h_i^{D_{\mathrm g}}$\\
\midrule
Karate & $(0.85,0.55)$ & $5$ & $5$ & $11$ & $17$
& & $4$ & $4$ & $0.52$\\
BA & $(0.85,0.30)$ & $7$ & $7$ & $11$ & $11$
& & $24$ & $26$ & $0.63$\\
\bottomrule
\end{tabular}
\caption{Source-placement results for the weighted Zachary karate-club network
\cite{zachary1977} ($n=34$, $m=78$, $\Delta=17$) and an unweighted
$\mathrm{BA}(120,2)$ graph ($n=120$, $m=236$, $\Delta=20$). Here $\delta$
and $\gamma$ are the exact discounted hitting domination and domination numbers,
respectively, and $D_{\mathrm g}$ is the greedy dominating set.}
\label{tab:experiments}
\end{table}

The results show that centrality orderings can be substantially less efficient
than direct optimization of the support floor. On Karate, HD reaches the
optimal cardinality of five, whereas degree and closeness require eleven and
seventeen sources, respectively. The five-source HD set is contained in both
threshold-reaching centrality prefixes, so the loss is due to selection order
rather than the absence of the relevant vertices from the rankings. Using
weighted strength instead of degree again requires eleven sources, while
weighted closeness based on inverse interaction weights requires eighteen.
The complete source-addition curves are provided in the reproducibility
repository.

Classical domination and discounted-floor feasibility are also incomparable.
On Karate, $\gamma=4$, and exhaustive enumeration gives exactly nine minimum
dominating sets; none reaches the floor $\tau=0.55$, and the largest
worst-vertex support among them is $0.523$. Thus adjacency to a source does
not by itself guarantee adequate discounted support. Conversely, domination
is not necessary. On the BA graph,
$\delta_{0.85,0.30}=7<24=\gamma$, so no minimum discounted-floor solution can
be a dominating set. The same holds on Karate: the greedy HD set
$\{0,1,5,23,33\}$, using the zero-based NetworkX labels, is
cardinality-optimal and feasible, although vertex $24$ has no selected vertex
in its closed neighborhood. Among the $55$ feasible five-source sets on
Karate, $36$ are dominating. Thus neither feasibility condition implies the
other.

These computations reflect the structural behavior already visible in
\cref{ex:not-distance,rem:center-not-always}. Discounted hitting support
depends on transition probabilities and on the full collection of paths to the
source set, not only on distance, degree, or adjacency coverage. On spiders,
\cref{thm:spider-algorithm} exploits this structure to obtain an exact
polynomial-time algorithm for fixed rational $(\lambda,\tau)$. On general
networks, \eqref{eq:milp} remains an exact formulation, while
\cref{thm:hardness} shows that, unless $\mathrm{P}=\mathrm{NP}$, no
polynomial-time algorithm exists for every fixed rational pair
$(\lambda,\tau)$.
\section{Concluding remarks}\label{sec:conclusion}

Discounted hitting domination turns a known discounted first-hitting potential
into a minimum-cardinality uniform-coverage problem. The hitting representation
\eqref{eq:hitting} yields monotone submodular support functions, while the
uniform floor admits both an exact submodular-cover formulation and an exact
mixed-integer linear formulation. The distance bounds recover distance-$r$
domination whenever
$\lambda^{r+1}<\tau\le(\lambda/\Delta)^r$, yielding NP-completeness and
APX-completeness at the fixed pair $(\lambda,\tau)=(1/4,1/14)$ on graphs of
maximum degree three. On spiders, the transfer formulas reduce the equilibrium
constraints to finitely many leg types and give an exact polynomial-time
algorithm for every fixed rational pair $(\lambda,\tau)$. The computational
examples show that optimal discounted-floor placements can differ sharply from
degree and closeness rankings and from classical domination.

Several directions remain open. One is whether the transfer approach extends
from spiders to caterpillars or to trees with a bounded number of branching
vertices. A related question appears in exponential domination: Dankelmann
et al.\ asked whether the exponential domination number of a tree can be
computed in polynomial time \cite{dankelmann2009}, and Bessy, Ochem, and
Rautenbach answered this affirmatively for subcubic trees \cite{bessy2016}.
A second direction is fault-tolerant discounted hitting domination, where the
support floor must survive the loss of one or more sources. Even on spiders,
source failure changes the transfer type seen at the center and requires
additional state information.

\section*{Data and code availability}

The computations in \cref{sec:experiments} use the weighted Zachary karate-club
network \cite{zachary1977} and an unweighted $\mathrm{BA}(120,2)$ graph with
seed $7$. Reproducibility code and pinned software dependencies are available
at
\href{https://github.com/aallagan/discounted-hitting-domination}
{\texttt{github.com/aallagan/discounted-hitting-domination}}.

The repository reproduces the mixed-integer computations of $\delta$ and
$\gamma$ in \cref{tab:experiments}, the exhaustive Karate checks, the
source-selection orders and support curves, and the weighted Karate baselines.
It also verifies selected theoretical results, including the distance-$r$
recovery identity, the spider formula of \cref{thm:spider-exact} against
exhaustive subset minimization, and the complete-graph formula using exact
rational arithmetic where applicable. An automated GitHub Actions workflow
installs the pinned dependencies and reruns the reproducibility and verification
checks.

\end{document}